\documentclass[12pt,twoside]{amsart}
 \title{Surfaces of globally $F$-regular type are of Fano type}

\author{Shinnosuke Okawa}

\date{\today}


\DeclareMathOperator{\Divi}{Div}
\DeclareMathOperator{\PPic}{\mathbf{Pic}}

\usepackage{mymacros}

\begin{document}
 
 \maketitle
 
\begin{abstract}
We prove that a projective surface of globally $F$-regular type defined over
a field of characteristic zero is of Fano type.
\end{abstract}

%
%------------------------
%

\section{Introduction}\label{sc:intro}
Projective varieties of Fano type form a natural and important class of varieties
from the point of view of the minimal model theory (see \pref{df:fano_type} for
definition). In this paper, we discuss a conjectural characterization of
this class proposed by Schwede and Smith:

\begin{conjecture}[{$=$\cite[Question 7.1]{MR2628797}}]\label{cj:Schwede_Smith_conjecture}
 Let $X$ be a normal projective variety defined over a field of characteristic zero.
 Then $X$ is of Fano type if and only if it is of globally $F$-regular type.
\end{conjecture}

For the context behind this conjecture, see \cite{MR2628797}.
The `only if' direction is already proven in \cite[Theorem 1.2]{MR2628797}, as an application of
\cite[Corollary 3.4]{MR2047704}. That is, a variety of Fano type
defined over a field of characteristic zero is
of globally $F$-regular type.
Aiming at the other direction,
the authors constructed a boundary divisor $\Delta$ for any globally $F$-regular variety $X$
(defined in a fixed positive characteristic)
such that $(X,\Delta)$ is log Fano \cite[Theorem 4.3 (i)]{MR2628797}.
The difficulty lies in that it is not clear if the boundary divisor $\Delta$ can be
lifted to characteristic zero or not.

On the other hand, it was proven in \cite[Theorem 1.2]{MR3275656} that
a Mori dream space of globally $F$-regular type is of Fano type.
Note that since a variety of Fano type in characteristic zero is a Mori dream space by \cite[Corollary 1.3.2]{MR2601039},
if \pref{cj:Schwede_Smith_conjecture} was solved affirmatively, then it would imply that a variety of globally $F$-regular type
is a Mori dream space.

The purpose of this paper is to prove \pref{cj:Schwede_Smith_conjecture} in dimension $2$,
\emph{without} assuming $X$ is a Mori dream space:

\begin{theorem}\label{th:SS_conjecture_for_surfaces}
Let $X$ be a projective surface defined over a field of characteristic zero.
Suppose that $X$ is of globally $F$-regular type. Then $X$ is of Fano type.
\end{theorem}

Here we describe the outline of our proof.
By \pref{lm:weak_Fano_is_of_Fano_type} and
\pref{lm:trace_back_mmp},
it is enough to show that $X$ admits a ($-K_X$)-MMP whose
final model is a weak Fano variety.
Note that in the proof of \cite[Theorem 1.2]{MR3275656},
the existence of such a ($-K_X$)-MMP was guaranteed by the assumption that
$X$ is a Mori dream space. 

To show the existence of such a nice MMP, we first check
in \pref{pr:globally_F-regular_surface_is_a_Mori_dream_space} that
a globally $F$-regular surface (defined in a fixed positive characteristic)
is a Mori dream space.
This implies that $X_{\mu}$, the reduction of our $X$ to a positive characteristic,
admits a ($-K_{X_{\mu}}$)-MMP.

Next we lift the ($-K_{X_{\mu}}$)-MMP to characteristic zero
to obtain a ($-K_X$)-MMP (strictly speaking we have to pass to
a base change of $X$ in this step, but it is a minor point).
Since $X_{\mu}$ is a surface, each step of the ($-K_{X_{\mu}}$)-MMP
is a divisorial contraction. Hence it is enough to show that the line bundle together with
the global sections which define the morphism lift to characteristic zero
(see \pref{pr:morphisms_lift}).
This follows from the standard deformation theory of line bundles (\pref{lm:line_bundles_lift})
and the vanishing theorem for cohomology of nef line bundles on globally $F$-regular varieties (\pref{lm:nef_vanishing_on_globally_F_regular_varieties}).
The rest of the argument is the same as that of the proof of \cite[Theorem 1.2]{MR3275656}.

Our proof does not use any classification result on surfaces
(the author would like to thank Yoshinori Gongyo
for pointing it out). In the end of this paper we clarify the obstructions which show up when we try
to use our approach in higher dimensions.

After the author finished writing a draft of this paper, which is a part of
the the author's PhD thesis, two relevant papers appeared.
Gongyo and Takagi gave another proof for the main theorem of this paper in \cite{gongyo2013surfaces},
and they also proved that surfaces of dense globally $F$-split type are of log Calabi-Yau type.
There is also the paper \cite{hwang2013characterization}, in which Hwang and Park gave yet another
proof for the main theorem of this paper based on results of Sakai \cite{MR761313}.

\subsection*{Acknowledgements}
The author would like to thank Osamu Fujino for informing him of the paper \cite{MR3319838}.
He would also like to thank Shunsuke Takagi for his useful comments.
He is grateful to the co-authors of the paper \cite{MR3275656}
for the discussion during the preparation of \cite{MR3275656}, which was quite helpful for him
to digest the contents of the paper.
During the preparation of this paper, the author was partially supported by
Grant-in-Aid for JSPS fellow 22-849.

%
%---------------------------------------------------
%

\section{Preliminaries}
For the notions of singularities in characteristic zero and Mori dream spaces
which will be used below without definition,
see \cite{MR1658959} and \cite{MR1786494} respectively.
Readers can also refer to
\cite[Sections 2.1 --- 2.3]{MR3275656}.

\begin{definition}\label{df:Mori_dream_space} 
A normal projective variety $X$ is called  a \emph{Mori dream space}
provided that the following conditions hold.
\begin{enumerate}[(1)]
\item $X$ is $\bQ$-factorial and $\PPic^0_X$ has dimension zero.  \label{it:fin_pic}
\item The nef cone of $X$ is the affine hull of finitely many semi-ample
line bundles. \label{it:nef_implies_sa}
\item There is a finite collection of small $\bQ$-factorial modifications
$
 f_i \colon X \dasharrow X_i
$
such that each $X_i$ satisfies (\pref{it:fin_pic})(\pref{it:nef_implies_sa}) and
the movable cone of $X$ is the union
of the nef cones of $X_i$. \label{it:movable_cone_condition}
\end{enumerate}
\end{definition}
The symbol $\PPic_X^0$ in (\pref{it:fin_pic}) stands for the identity component of
the Picard scheme of $X$ over $\bfk$
(see \cite[Chapter 9]{MR2222646} for details).
To make sense of \eqref{it:movable_cone_condition},
note that the Weil divisors on $X_i$ are identified with those on $X$
via the small birational map $f_i$.

Consider a normal projective variety
$
 X
$
and an effective $\bQ$-divisor $\Delta$ on $X$
such that
$
 K_ X + \Delta
$
is $\bQ$-Cartier. Such a pair
$
 ( X, \Delta )
$
is called a (\emph{log})\emph{pair}.

\begin{definition}\label{df:fano_type}
A pair $(X,\Delta)$ is called a \emph{log Fano pair} if
its singularity is at worst klt and $- ( K_X + \Delta )$ is ample.
A normal projective variety $X$ is said to be \emph{of Fano type} if
there exists an effective $\bQ$-divisor $\Delta$ on $X$ such that
$(X,\Delta)$ is a log Fano pair.
$
 X
$
is called a \emph{weak Fano variety}
if it admits at worst log terminal singularities and
$
 - K _X
$
is nef and big.
\end{definition}

We will repeatedly use
\begin{lemma}[{$=$ \cite[Remark 2.8]{MR3275656}}]\label{lm:weak_Fano_is_of_Fano_type}
A weak Fano variety is a variety of Fano type.
\end{lemma}

The following claim, whose proof is implicitly given in the proof of
\cite[Theorem 3.2]{MR3275656}, will be repeatedly used in this paper.

\begin{lemma}\label{lm:trace_back_mmp}
Let
\begin{equation}
 X \dashto X _{1} \dashto \cdots \dashto X _{\ell}
\end{equation}
be a
$
 ( - K _X )
$-MMP such that
$
 X _{ \ell }
$
is a variety of Fano type.
Then so is
$
 X
$.
\end{lemma}

Next we explain about singularities in positive characteristics.
Let $R$ be a commutative algebra defined over a field of characteristic $p>0$
such that the (absolute) Frobenius map
$
 F \colon R \to R
$
is finite. By an abuse of notation, the corresponding morphism between the
spectra will be also denoted by
\begin{equation}
 F \colon \Spec R \to \Spec R.
\end{equation}
The finiteness of $F$ is preserved under morphisms of finite type, and hence
it comes for free if we are dealing with algebraic schemes over a perfect field.
For an $R$-module $M$, the $e$-times pushforward
$
 F^e_*M
$
is isomorphic to $M$ as an abelian group, whereas the action of $R$
is twisted as
\begin{equation}
 r\cdot F^e_*m = F^e_* ( r^{p^e} m ) \in F^e_*M.
\end{equation}
In the rest of paper, the element
$
 F^e_* m \in F^e_* M
$
will be simply denoted by
$
 m
$
if it can be read off from the context.

\begin{definition}\label{df:Frobenius}
\begin{enumerate}[(i)]
\item
An $F$-finite ring $R$ of characteristic $p>0$
is \emph{strongly $F$-regular} if for every $c \neq 0 \in R$,
there exists an integer $e \ge 1$ such that the homomorphism of $R$-modules
\begin{equation}
 cF^e \colon R \xto[]{F^e} F^e_*R \xto[]{F^e_*(c \cdot)} F^e_* R
 \ ; \quad a \mapsto c a^{p^e}
\end{equation} 
splits (i.e. admits a left inverse) as a homomorphism of $R$-modules.
 
\item Let $X$ be a variety defined over an $F$-finite field of characteristic $p>0$.
We say that $X$ is \emph{globally $F$-regular} if for every effective divisor $D$ on $X$,
there exists an integer $e \ge 1$ such that
\begin{equation}
 s_DF^e \colon \cO_X \to F^{e}_{*}(\cO_X(D)); \quad a \mapsto a^{p^e}s_D
\end{equation}
splits as a morphism of $\cO_X$-modules, where
$
 s_D \in H^0(X, \cO_X(D))
$
is the global section of $\cO_X(D)$ defined by the effective divisor $D$.
\item Let $X$ be a variety defined over a field of characteristic zero.
$X$ is said to be \emph{of globally $F$-regular type} if
there exists a model
$
 \pi \colon X_A \to \Spec {A}
$
of $X$
such that
$
 X_{\mu} = X_A \otimes_A k(\mu)
$
is globally $F$-regular for any closed point
$
 \mu \in \Spec A
$
(see \cite[page 7]{MR3275656} for the definition of models).
A ring $R$ of finite type over a field of characteristic zero
is of strongly $F$-regular type if
$
 X = \Spec A
$
is of globally $F$-regular type.
\end{enumerate}
\end{definition}

Globally $F$-regular varieties share nice properties with varieties of Fano type.
In particular, we will use the following
\begin{lemma}[{$=$ \cite[Corollary 4.3]{MR1786505}}]\label{lm:nef_vanishing_on_globally_F_regular_varieties}
Let
$
 X
$
be a globally $F$-regular projective variety.
Then for any nef line bundle
$
 L
$
on
$
 X
$,
the vanishing
\begin{equation}
 H^i ( X, L ) = 0
\end{equation}
holds for any
$
 i > 0
$.
\end{lemma}

In order to illustrate how the definition of global $F$-regularity
is used in practice, we include a proof.

\begin{proof}
Pick an ample line bundle
$
 A
$
on
$
 X
$.
Then
$
 B_n = A \otimes L ^{ \otimes n }
$
is ample for any
$
 n > 0.
$
Take
$
 q = p^e
$
such that
$
 F^e \colon \cO _X \to F^e_* \cO_X
$
splits. Taking tensor product with
$
 B_n
$,
we see that
$
 B_n
$
is a direct summand of
$
 F^e_* \cO_X \otimes B_n \cong F^e_* (B_n ^{ \otimes q }).
$
Thus we obtain the sequence of inclusions
$
 H^i ( X, B_n ) \subset H^i ( X, B_n ^{ \otimes q } ) \subset H^i ( X, B_n ^{ \otimes q^2 } )
 \subset \cdots
$,
and this should vanish eventually by the Serre's vanishing theorem.

Now assume
$
 A
$
is effective.
So far we checked
$
 H^i ( X, A \otimes L ^{ \otimes n } ) = 0
$
for all
$
 n > 0.
$
On the other hand, we know
$
 s_A F^e \colon \cO _X \to F^e_* \cO_X
$
splits for some
$
 e > 0
$.
Taking tensor product with
$
 L
$,
we obtain the inclusion
$
 H^i ( X, L ) \subset H^i ( X, A \otimes L ^{ \otimes q } )
$
to conclude the proof.
\end{proof}

\begin{remark}\label{rm:Q_factoriality}
For the sake of completeness, here we include a remark
on $\bQ$-factoriality of the varieties we are interested in.
Log terminal singularities and strongly
$F$-regular singularities are rational singularities.
Hence the $\bQ$-factoriality of surfaces with those singularities
comes for free by \cite[Proposition 17.1]{MR0276239}.
\end{remark}

The following proposition is the first step toward the proof of \pref{th:SS_conjecture_for_surfaces}.

\begin{proposition}\label{pr:globally_F-regular_surface_is_a_Mori_dream_space}
Let $X$ be a projective and globally $F$-regular surface defined over an $F$-finite field.
Then $X$ is a Mori dream space.
\end{proposition}

\begin{proof}
We directly check the conditions of \pref{df:Mori_dream_space}.
First of all, see \pref{rm:Q_factoriality} for the $\bQ$-factoriality.
Since
$
 H^1 (X, \cO_X)
$,
the tangent space of $\PPic^0_X$ at the origin,
is trivial by \pref{lm:nef_vanishing_on_globally_F_regular_varieties}, the condition \eqref{it:fin_pic} is satisfied.
To see \eqref{it:nef_implies_sa}, note that by \cite[Theorem 4.3 (i)]{MR2628797} there exists an effective
$\bQ$-divisor $\Delta$ on $X$ such that $(X,\Delta)$ is a log Fano pair.
By the cone and contraction theorems for surface pairs in positive characteristics
(see \cite[Theorem 3.13]{MR3319838} and \cite[Theorem 3.21]{MR3319838} respectively),
we see that the effective cone (hence the nef cone) of $X$ is rational polyhedral and that
every extremal ray of the nef cone is generated by a semi-ample line bundle
which defines the extremal contraction of the supporting facet of the effective cone.
Finally, since $X$ is a surface, any movable divisor is semi-ample by the Zariski-Fujita theorem.
Therefore the condition \eqref{it:movable_cone_condition} is automatically satisfied.
\end{proof}

We next prove the liftability of morphisms from globally $F$-regular varieties. 

\begin{lemma}\label{lm:line_bundles_lift}
Let $A$ be a complete local ring with the closed point
$
 \mu \in \Spec{A}
$,
and
$
 f \colon X_A \to \Spec{A} 
$
a flat projective morphism.
Suppose
$
 H^{2} ( X_{\mu}, \cO_{X_{\mu}}) = 0,
$
where
$
 X_{\mu}
$
is the closed fiber.
Then any line bundle $L_{\mu}$ on $X_{\mu}$ extends to a line bundle $L_A$ on $X_A$
so that $L_A|_{X_{\mu}}\simeq L_{\mu}$.
\end{lemma}

\begin{proof}
See \cite[Corollary 8.5.6(a)]{MR2222646}.
\end{proof}

%The following proposition is the key ingredient of the proof of our main theorem.

\begin{proposition}\label{pr:morphisms_lift}
Let $A$ and
$
 f \colon X_A \to \Spec{A}
$ be as in \pref{lm:line_bundles_lift}.
Assume that $X_{\mu}$ is globally $F$-regular.
Suppose that there exists a morphism with connected fibers
$
 g_{\mu} \colon X_{\mu} \to Y_{\mu}
$
to a normal projective variety $Y_{\mu}$.
Then $g_{\mu}$ lifts to a morphism with connected fibers
$
 g_A \colon X_A \to Y_A
$
over
$A$,
where $Y_A$ is a normal projective scheme over $A$.
\end{proposition}

\begin{proof}
Pick a very ample line bundle
$
 \cO _{ Y _{ \mu } } (1)
$
on
$
 Y _{ \mu }
$, so that the complete linear system corresponding to the line bundle
$
 L_{\mu} = g _{ \mu }^* \cO _{ Y _{ \mu } } (1)
$
defines the morphism $g_{\mu}$: to see this use the equality
$
 H^0 ( X _{ \mu }, L _{ \mu } ) = H^0 ( Y _{ \mu }, \cO _{ Y _{ \mu } } (1) )
$,
which follows from the assumption that
$
 g _{ \mu }
$
has connected fibers.

Since $X_{\mu}$ is globally $F$-regular,
we see
$
 H^{2}(X_{\mu},\cO_{X_{\mu}}) = 0
$
by \pref{lm:nef_vanishing_on_globally_F_regular_varieties}.
Hence by \pref{lm:line_bundles_lift} we obtain a line bundle $L_A$ on $X_A$
which is a lift of $L _{ \mu }$.

Also, since $ L _{ \mu } $ is nef, we see
$
 H^{1}(X_{\mu},L_{\mu}) = 0
$
again by \pref{lm:nef_vanishing_on_globally_F_regular_varieties}.
Then \cite[Chapter III, Corollary 12.9]{Hartshorne} implies $H^{1}(X_A,L_A)=0$.
Applying \cite[Chapter III, Theorem 12.11 (b)]{Hartshorne} for $i=1$ and then
\cite[Chapter III, Theorem 12.11 (a)]{Hartshorne} for $i=0$, we can check that
$H^{0}(X_A,L_A)$ is a free $A$ module satisfying
\begin{equation}
 H^{0} ( X_A, L_A ) \otimes _{A} k (\mu)
 \simto H^{0} ( X_{\mu}, L_{\mu} ).
\end{equation}
Since
$
 L _{ \mu }
$
is globally generated and the base locus of a linear system is closed,
this implies that $L_A$ is globally generated and that the corresponding complete
linear system defines a morphism
$
 g'_A \colon X_A \to Y'_A
$, which restricts to the morphism $g_{\mu}$ over
$
 \mu \in \Spec A.
$
By taking the Stein factorization
$
 g_A \colon X_A \to Y_A
$
of $g'_A$, we obtain the desired lift of $g _{ \mu }$.
\end{proof}

In the final step of the proof of \pref{th:SS_conjecture_for_surfaces}, we use the following
fact that the property of being of Fano type descends under base field extension.
\begin{proposition}\label{pr:Mori_dream_space_and_base_change}
Let $X$ be a normal projective variety defined over a field $k$, and
$
 k \subset K
$
an extension of fields.
If the base change $X_K=X\otimes_{k}K$ is a Mori dream space, so is
$X$.
\end{proposition}

\begin{proof}
Note first that the condition
$
 \dim \PPic^0 = 0
$
is stable under the change of base field, because of the functorial behaviour
$
 \PPic _{ X_K / K }
 \cong
 \PPic _{ X_k / k } \otimes _{ k } K
$.
Hence by the characterization of Mori dream spaces
\cite[Proposition 2.9]{MR1786494},
it is enough to show that a Cox ring of $X$ is of finite type over $k$.
%;for the definition and notation of multi-section rings and Cox rings, see \cite[Definition 2.18]{MR3275656}.

Let $\Gamma\subset\Divi{X}$ be a finitely generated group of Cartier divisors on $X$
which defines a Cox ring of $X$. Then we have the canonical isomorphism
\begin{equation}\label{eq:base_change_of_cox_ring}
 R ( X, \Gamma ) \otimes_k K \simto R( X_K, \Gamma_K ),
\end{equation}
where $\Gamma_K$ is the pull-back of $\Gamma$ to $X_K$.

As explained in \cite[Lemma 3.2]{okawa2011images},
the multi-section ring on the RHS of
\eqref{eq:base_change_of_cox_ring} is of finite type over $K$.
Using the descent of finite generation with respect to the field extension
$
 k \subset K
$,
we obtain the finite generation of $R(X, \Gamma)$.
\end{proof}

\begin{corollary}\label{cr:Fano_type_and_base_change}
Let $X$ be a normal projective variety defined over a field $k$ of characteristic zero,
and
$
 k \subset K
$
an extension of fields.
If the base change
$
 X_K = X \otimes_{k} K
$
is of Fano type, so is
$X$.
\end{corollary}
\begin{proof}
By \cite[Corollary 1.3.2]{MR2601039} and
\pref{pr:Mori_dream_space_and_base_change},
$X$ is a Mori dream space.
Since $-K_{X_K}$ is big, so is $-K_X$. Hence we find a ($-K_X$)-MMP
$X=X_0\dasharrow X_1\dasharrow \cdots \dasharrow X_{\ell}$ such that $-K_{X_{\ell}}$ is semi-ample
and big by \cite[Proposition 1.11(1)]{MR1786494}.
Taking the base change by $k\subset K$, we get a ($-K_{X_K}$)-MMP.
Since $X_K$ is of Fano type, so are
$
 X_i \otimes_k K
$
by \cite[Lemma 3.1]{MR3275656} and \cite[Theorem 5.1]{MR2944479}.
Hence $X_{\ell}$ also has at worst log terminal singularities, which means that
$X_{\ell}$ is a weak Fano variety.
Finally we can use
\pref{lm:trace_back_mmp}
to conclude the proof.
\end{proof}

%
%-------------------------------------
%

\section{Proof of \pref{th:SS_conjecture_for_surfaces}}

\begin{proof}[Proof of \pref{th:SS_conjecture_for_surfaces}]

Let
$
 \pi \colon X_A \to \Spec{A}
$
be a model of $X$, so that
$X_{\mu}$ is globally $F$-regular for a closed point
$
 \mu \in \Spec{A}
$.
By \pref{pr:globally_F-regular_surface_is_a_Mori_dream_space},
$X_{\mu}$ is a Mori dream space.

Since $X_{\mu}$ is a Mori dream space, there exists a
$
 ( - K _{ X _{\mu} } )
$-MMP
\begin{equation}\label{eq:MMP_in_positive_characteristic}
 X_{\mu} = X _{\mu, 0} \to X _{\mu, 1} \to \cdots \to X _{\mu, \ell}
\end{equation}
such that
$
 - K _{ X _{ \mu, \ell } }
$
is semi-ample and big by \cite[Proposition 1.11(1)]{MR1786494}.
Since each morphism is a projective morphism with connected fibers
between normal varieties, which in particular is an algebraic fiber space,
we see that
$X_{\mu, i}$ are also globally $F$-regular by
\cite[Lemma 2.14]{MR3275656}.

Let $A_{\mu}$ be the localisation of $A$ at $\mu$.
Consider its completion
$\Ahat$ at the maximal ideal, and
take the base change of
$
 X_A \to \Spec{A}
$
by
$
 A \subset \Ahat
$.
In the rest, we will write
$
 A
$
to mean
$
 \Ahat
$.

The generic (resp. closed) point of
$
 \Spec A ( = \Ahat )
$
will be denoted by $\xi$ (resp. $\mu$).
Note that the generic fiber
$
 X_{\xi}
$
is still of globally $F$-regular type.

By \pref{pr:morphisms_lift}, the morphisms in
\eqref{eq:MMP_in_positive_characteristic}
lift over
$
 \Spec { A }
$.
They will be denoted by
$
 X_{A} = X_{A,0} \to X_{A,1} \to \cdots \to X_{A,\ell}
$.
Since
$
 - m K_{X_{\mu, \ell}}
$ is globally generated for sufficiently divisible $m$,
by the same arguments as in the proof of \pref{pr:morphisms_lift},
we see that all the global sections of
$
 - m K _{ X _{ \mu, \ell } }
$
lift over $\Spec{A}$.
Thus, restricting to the generic fiber, we obtain a $(-K_{X_{\xi}})$-MMP
$
 X _{ \xi }
 =
 X _{ \xi, 0 } \to X _{ \xi, 1 }
 \to \cdots \to X _{ \xi, \ell }
$
such that
$
 - K _{ X _{ \xi, \ell } }
$
is semi-ample and big.

Using \cite[Lemma 2.14]{MR3275656}, we see that
$
 X _{ \xi, \ell }
$
is also of globally $F$-regular type.
Hence by
\cite[Proposition 2.17]{MR3275656},
it has only log terminal singularities.
Thus we see that it is a weak Fano variety.
Finally we can use
\pref{lm:weak_Fano_is_of_Fano_type} and
\pref{lm:trace_back_mmp}
to conclude the proof.
\end{proof}

%
%--------------------------------------------------------------
%

\section{Concluding discussion: toward higher dimensions}
When we try to generalize our result to higher dimensions,
the following issues show up.
First of all, the existence of an anti-canonical MMP was essential in our arguments.
For that purpose we have to settle the following question affirmatively.

\begin{question}\label{qs:globally_F-regular_variety_is_a_Mori_dream_space}
Is a projective and globally $F$-regular variety over a finite field always a Mori dream space?
\end{question}

Now let $X$ be a projective variety of globally $F$-regular type.
If \pref{qs:globally_F-regular_variety_is_a_Mori_dream_space} is answered affirmatively,
then we obtain an anti-canonical MMP for a reduction $ X _{ \mu } $ of $X$ to
a positive characteristic ending up with a weak log Fano model.
Suppose that
$
 f_{\mu,n} \colon X_{\mu,n}\dasharrow X_{\mu,n+1}
$
is a step of such an anti-canonical MMP.
If it is a divisorial contraction, we can lift it to characteristic zero
as before.

Suppose that it is a $(-K_{X_{\mu,n}})$-flip.
In this case we can still lift the flipping contraction, say
$
 X _{ \mu, n } \to Y _{ \mu }
$, to characteristic zero.
Let $g_A\colon X_{A, n}\to Y_A$ be the $(-K_{X_{A,n}})$-flipping contraction thus obtained.
Now the question is if one can lift the flip
$
 X_{\mu, n+1} \to Y _{ \mu }
$
to characteristic zero as well.

The existence of the $(-K_{X_{\mu,n}})$-flip is equivalent to
the finite generation of the
$
 \cO _{ Y _{ \mu } }
$-algebra
\begin{equation}\label{eq:known_finite_generation}
 \bigoplus_{m \ge 0}g_{\mu,n*}\cO_{X_{\mu,n}}(-mK_{X_{\mu,n}})
 = \bigoplus_{m\ge 0}\cO_{Y_{\mu}}(-mK_{Y_{\mu}}).
\end{equation}
The corresponding statement for the existence of
$(-K_{A,n})$-flip holds as well.

If one could show that the restrictions of the sheaves
$
 \omega_{Y_A}^{[m]} = ( \omega _{ Y _{ A } }^m ) ^{ ** }
$
to $Y_{\mu}$ are reflexive for all $m$ which is divisible by a fixed integer $N > 0$,
then that would imply that the natural homomorphism
\begin{equation}
 \bigoplus_{m\ge 0, \ N|m}\cO_{Y_A}(-mK_{Y_A})\to
\bigoplus_{m\ge 0, \ N|m}\cO_{Y_{\mu}}(-mK_{Y_{\mu}})
\end{equation}
is surjective. Then, using the Nakayama's lemma, we can derive the desired finite generation
from that of \eqref{eq:known_finite_generation}.

Unfortunately the author does not see how to verify the necessary reflexivity.

%
%--------------------------------------------------------------------------------
%Bibliography

\newcommand{\etalchar}[1]{$^{#1}$}
\def\cprime{$'$} \def\cprime{$'$}
\providecommand{\bysame}{\leavevmode\hbox to3em{\hrulefill}\thinspace}
\providecommand{\MR}{\relax\ifhmode\unskip\space\fi MR }
\providecommand{\MRhref}[2]{%
  \href{http://www.ams.org/mathscinet-getitem?mr=#1}{#2}
}
\providecommand{\href}[2]{#2}

Department of Mathematics,
Graduate School of Science,
Osaka University,
Machikaneyama 1-1,
Toyonaka,
Osaka,
560-0043,
Japan.

{\em e-mail address}\ : \  okawa@math.sci.osaka-u.ac.jp

\end{document}